\documentclass[12pt]{article}
\usepackage{amsmath}
\usepackage{amsthm}
\usepackage{amssymb}
\usepackage{hyperref}
\usepackage{authblk}
\theoremstyle{plain}
\usepackage{cite}
\usepackage{graphicx}
\newtheorem{theorem}{Theorem}[section]
\newtheorem{lemma}[theorem]{Lemma}

\theoremstyle{definition}
\newtheorem{definition}[theorem]{Definition}
\newtheorem{example}[theorem]{Example}

\theoremstyle{remark}

\title{Numerical Approximation of Stochastic Volterra-Fredholm Integral Equation using Walsh Function.}
\author{Prit Pritam Paikaray \thanks{paikaraypritam@gmail.com}}
\author{Sanghamitra Beuria\thanks{sbeuria108@gmail.com}}
\author{Nigam Chandra Parida\thanks{ncparida@gmail.com}}
\affil[1]{Department of Mathematics, College of Basic Science and Humanities, OUAT\\
	Bhubaneswar, Odisha, 751003,
	India
}

\begin{document}
	\maketitle
		\begin{abstract}
		In this paper, a computational method is developed to find an approximate solution of the stochastic Volterra-Fredholm integral equation using the Walsh function approximation and its operational matrix. Moreover, convergence and error analysis of the method is carried out to strengthen the validity of the method. Furthermore, the method is numerically compared to the block pulse function method and the Haar wavelet method for some non-trivial examples.\\
		\textbf{Keywords:} Stochastic Volterra-Fredholm integral equation, Brownian motion, It$\hat{o}$ integral, Walsh approximation, Lipschitz condition.\\
		\textbf{AMS2010 classification:} 60H05, 60H35, 65C30.
	\end{abstract}

	\section{Introduction} 
Stochastic differential equations (SDE) are widely used in a variety of fields, including the physical sciences, biological sciences, agricultural sciences, and financial mathematics, which includes option pricing, where stochastic Volterra-Fredholm integral equation (SVFIE) plays a crucial role \cite{Oksendal,Kloeden,tudor}. As with other differential equations, it is practically impossible to find the solution to many SDE, and the problem becomes more complex in the case of SVFIE. Therefore, numerical approximation method becomes crucial for finding the solutions to the problems. Numerous SVFIEs are determined approximately using a variety of numerical techniques.
In recent decades, orthogonal functions such as block pulse function (BPF), Haar wavelet, Legendre polynomials, Laguerre polynomials, and Chebyshev's polynomials have been utilised to approximate the solution of SVFIE.

The Walsh functions form an orthonormal system that only accepts the values $1$ and $-1$. As a result, a number of mathematicians consider the Walsh system to be an artificial orthonormal system, which was introduced in 1923 citeWalsh and has numerous applications in digital technology. Walsh functions have a significant advantage over traditional trigonometric functions because a computer can determine the exact value of any Walsh function at any given time with high accuracy. Chen and Hsiao used the Walsh function to solve the variational problem in 1997, is cited in \cite{Chen}. They used the same concept to solve the integral equation \cite{Hsiao} in 1979. The technique's key property is that it converts the problem into a system of algebraic equations, which are then solved to yield an approximation of the solution. In this paper, we use the Walsh function to approximate the solution $x(t)$ of the following linear SVFIE 

\begin{equation}\label{SVIE}	
	x(t)=f(t)+\int_{\alpha}^{\beta}k(s,t)x(s)ds+\int_{0}^{t}k_1(s,t)x(s)ds+\int_{0}^{t}k_2(s,t)x(s)dB(s)	
\end{equation}

where $x(t)$, $f(t)$, $k(s,t)$, $k_1(s,t)$ and $k_2(s,t)$ for $s,t\in[0,T)$, represent the stochastic processes primarily based on the identical probability space $(\Omega,F,P)$ and $x(t)$ is unknown. In addition, $B(t)$ represents Brownian motion \cite{Kloeden,Oksendal}, and $\int_0^{t}k_2(s,t)x(s)dB(s)$ represents the It$\hat{o}$ integral.

In the majority of previous works, the evaluation is predicated on the assumption that the derivatives$f'(t)$, $\frac{\partial^2 k}{\partial s \partial t}$, $\frac{\partial^2 k_i}{\partial s \partial t}$ for $i=1, 2$, exists and bounded. By converting BPF approximation to Walsh function approximation in this paper, we expect only Lipschitz continuity of the functions $f(t), k(s,t), k_1(s, t)$ and $k_2 (s, t)$ to have the same rate of convergence, which allows us to consider the general form of SVFIE to be integrated. In the final portion, the method is compared to similar techniques \cite{Khodabin,Mohamaddi} that approximate the solution of the SVFIE using block pulse function and Haar wavelet.
\section{Walsh Function and its Properties}\label{Walsh}
\begin{definition}[Rademacher Function]
	Rademacher function $r_i(t)$, $i=1,2,\hdots$, for $t\in[0,1)$ is defined  by \cite{Walsh} 
	
	$$r_i(t)=
	\Biggl\{
	\begin{aligned}
		1 & \quad \text{ $i=0$},\\
		sgn(sin(2^i\pi t)) & \quad \text{otherwise}\\
	\end{aligned}$$
	where,
	$$sgn(x)=
	\Biggl\{
	\begin{aligned}
		1 & \quad \text{ $x>0$},\\
		0 & \quad \text{ $x=0$},\\
		-1 & \quad \text{$x<0$}.
	\end{aligned}$$
\end{definition}
\begin{definition}[Walsh Function]
	The $n^{th}$ Walsh function for $n=0,1,2,\cdots,$ denoted by $w_n(t)$, $t\in[0,1)$ is defined \cite{Walsh} as 
	$$w_n(t)=(r_q(t))^{b_q}.(r_{q-1}(t))^{b_{q-1}}.(r_{q-2}(t))^{b_{q-2}}\hdots (r_{1}(t))^{b_1}$$ where $n=b_q2^{q-1}+b_{q-1}2^{q-2}+b_{q-2}2^{q-3}+\hdots +b_12^{0}$ is the binary expression of $n$. Therefore, $q$, the number of digits present in the binary expression of $n$ is calculated by $q=\big[\log_2n\big]+1$ in which $\big[\cdot\big]$ is the greatest integer less than or equal's to $'\cdot'$.
\end{definition}

The first $m$ Walsh functions for $m \in \mathbb{N}$ can be written as an $m$-vector by
$W(t)=\begin{bmatrix}
	w_0(t) & w_1(t) &w_2(t)\hdots w_{m-1}(t)
\end{bmatrix}^T$, $t\in[0,1)$. The Walsh functions satisfy the following properties.

\subsection*{Orthonormality}
The set of Walsh functions is orthonormal. i.e., 
$$\int_{0}^{1}w_i(t)w_j(t)dt=\Biggl\{\begin{aligned}
	1 & \quad \text{i=j,}\\
	0 & \quad \text{otherwise}.
\end{aligned}$$
\subsection*{Completeness}
For every $f\in L^2([0,1))$ 
\begin{equation*}
	\int_{0}^{1}f^2(t)dt=\sum_{i=0}^{\infty}f_i^2\lvert\lvert w_i(t)\lvert\lvert^2
\end{equation*}
where $f_i=\int_{0}^{1}f(t)w_i(t)dt$.
\subsection*{Walsh Function Approximation}
Any real-valued function $f(t)\in L^2([0,1))$ can be approximated as 
$$f_m(t)=\sum_{i=0}^{m-1}c_iw_i(t)$$
where, $c_i=\int_{0}^{1}f(t)w_i(t)dt$.\\
The matrix form is given by
\begin{equation}
	f(t)=F^TT_WW(t) \label{Eq:2}
\end{equation}where
$ F=
\begin{bmatrix}
	f_0 &f_1 &f_2 \hdots f_{m-1}
\end{bmatrix}^T$, 
$f_i=\int_{ih}^{(i+1)h}f(s)ds$.\\
Here, $T_W=[w_i(\eta_j)]$ is called as the Walsh operational matrix where $\eta_j\in [jh, (j+1)h)$.\\

Similarly, function $k(s,t)\in L^2([0,1)\times[0,1))$ can be approximated by
$$k_m(s,t)=\sum_{i=0}^{m-1}\sum_{j=0}^{m-1}c_{ij}w_i(s)w_j(t)$$
where, $c_{ij}=\int_{0}^{1}\int_{0}^{1}k(s,t)w_i(s)w_j(t)dtds$
with the matrix form represented by
\begin{equation}
	k(s,t)=W^T(s)T_WKT_WW(t)=W^T(t)T_WK^TT_WW(s) \label{Eq:3}
\end{equation}
where $K=[k_{ij}]_{m\times m}, k_{ij}=\int_{ih}^{(i+1)h}\int_{jh}^{(j+1)h}k(s,t)dtds$.
\section{Relationship between Walsh Function and Block Pulse Functions (BPFs)}
\begin{definition}[Block Pulse Functions]
	For a fixed positive integer $m$, an $m$-set of BPFs $\phi_i(t), t\in [0,1)$ for $i=0, 1,..., m-1$ is defined as
	$$\phi_i(t)=\biggl\{
	\begin{aligned}
		1 & \quad \text{if $\frac{i}{m}\le t < \frac{(i+1)}{m}, \quad$}\\
		0 & \quad \text{ otherwise}
	\end{aligned}
	$$
	$\phi_i$ is known as the $i$th BPF.
\end{definition}
The set of all $m$ BPFs can be written concisely as an $m$-vector,
$\Phi(t)=\begin{bmatrix}
	\phi_0(t) & \phi_1(t) &\phi_2(t)\hdots \phi_{m-1}(t)
\end{bmatrix}^T$, $t\in[0,1)$.

The BPFs are disjoint, complete, and orthogonal \cite{Ganti}.

The BPFs in vector form satisfy 
$$ \Phi(t)\Phi(t)^TX=\tilde{X}\Phi(t) \;\textrm{and}\; \Phi^T(t)A\Phi(t)=\hat{A}\Phi(t)$$
where $X \in \mathbb{R}^{m \times 1}, \tilde{X}$ 
is the $m\times m$ diagonal matrix with $\tilde{X}(i, i)=X(i) \,\textrm{for}\, i=1, 2, 3\cdots m, A\in \mathbb{R}^{m \times m}$ and  $\hat{A}=\begin{bmatrix}
	a_{11}&	a_{22}&\hdots 	&a_{mm}
\end{bmatrix}^T$
is the $m$-vector with elements equal to the diagonal entries of $A$.
The integration of BPF vector $\Phi(t)$, $t\in[0,1)$ can be performed by \cite{Hatamzadeh}
\begin{equation}
	\int_{0}^{t}\Phi(\tau)d\tau=P\Phi(t), t\in[0,1), 
\end{equation}
Hence, the integral of every function $f(t)\in L^2[0,1)$ can be approximated as$$\int_{0}^{t}f(s)ds=F^TP\Phi(t)$$
The It$\hat{o}$ integral of BPF vector $\Phi(t)$, $t\in[0,1)$ can be performed by \cite{Maleknejad}
\begin{equation}
	\int_{0}^{t}\Phi(\tau)dB(\tau)=P_S\Phi(t), t\in[0,1)
\end{equation}
Hence, the It$\hat{o}$ integral of every function $f(t)\in L^2[0,1)$ can be approximated as $$\int_{0}^{t}f(s)dB(s)=F^TP_S\Phi(t).$$

The following theorem describes a relationship between the Walsh function and the block pulse function.
\begin{theorem}\cite{Paikaray}
	Let the $m$-set of Walsh function and BPF vectors be $W(t)$ and $\Phi(t)$ respectively. Then the BPF vectors $\Phi(t)$ can be used to approximate $W(t)$ as $W(t)=T_W\Phi(t)$, $m=2^k$, and $k=0,1,\hdots $, where $T_W=\big[c_{ij}\big]_{m\times m}$, $c_{ij}=w_i(\eta_j)$, for some  $\eta_j=\big(\frac{j}{m},\frac{j+1}{m}\big)$ and $i,j=0,1,2,\hdots m-1$.
\end{theorem}

One can see that \cite{Cheng}  $$T_WT_W^T=mI \, \textrm{and} \, T_W^T=T_W$$ which implies that $\Phi(t)=\frac{1}{m}T_WW(t).$\\
\begin{lemma}[Integration of Walsh function]
			Suppose that $W(t)$ is a Walsh function vector, then the integral of $W(t)$ w.r.t. $t$ is given by \\
$\int_{0}^{t}W(s)ds=\wedge W(t)$, where $\wedge =\frac{1}{m}T_WPT_W$ and $$P=\frac{1}{h}\begin{bmatrix}
	1 &2 &2&\hdots &2\\0 &1 &2&\hdots &2\\\vdots &\vdots &\vdots &\ddots &\vdots\\0 &0 &0&\hdots &1
\end{bmatrix}$$
\end{lemma}

\begin{lemma}[Stochastic integration of Walsh function]\cite{Paikaray}
			Suppose that $W(t)$ is a Walsh function vector, then the It$\hat{o}$ integral of $W(t)$ is given by\\
$\int_{0}^{t}W(s)dB(s)=\wedge _S W(t)$, where $\wedge_S =\frac{1}{m}T_WP_ST_W$ and $$P_S=
\begin{bmatrix}
	B(\frac{h}{2}) &B(h)&\hdots &B(h)\\0 &B(\frac{3h}{2})-B(h)&\hdots &B(2h)-B(h)\\ \vdots &\vdots &\ddots &\vdots \\0 &0&\hdots &B(\frac{(2m-1)h}{2})-B((m-1)h)
\end{bmatrix}.$$
\end{lemma}
\section{Numerical Solution of  Stochastic Volterra- Fredholm Integral Equation}
We consider following linear Stochastic Volterra-Fredholm Integral equation(LSVFIE)
\begin{equation}\label{Eq:Ram}
	x(t)=f(t)+\int_{0}^{1}k(s,t)x(s)ds+\int_{0}^{t}k_1(s,t)x(s)ds+\int_{0}^{t}k_2(s,t)x(s)dB(s)
\end{equation}
where $x(t)$, $f(t)$, $k(s,t)$, $k_1(s,t)$ and $k_2(s,t)$ for $s,t\in[0,T)$, are the stochastic processes defined on the same probability space $(\Omega,F,P)$ and $x(t)$ is unknown. Also $B(t)$ is Brownian motion and $\int_{0}^{t}k_2(s,t)x(s)dB(s)$ is the It${o}$ Integral.\\
Using equation \eqref{Eq:2} and\eqref{Eq:3} in \eqref{Eq:Ram} we have
\begin{eqnarray}
	X^TT_WW(t)&=&F^TT_WW(t)+\int_{0}^{1}W^T(t)T_WK^TT_WW(s)W^T(s)T_WXds\nonumber\\
	&&+\int_{0}^{t}W^T(t)T_WK^T_1T_WW(s)W^T(s)T_WXds\nonumber\\
	&&+\int_{0}^{t}W^T(t)T_WK^T_2T_WW(s)W^T(s)T_WX dB(s)\nonumber\\
	&=&F^TT_WW(t)+W^T(t)T_WK^TT_W\int_{0}^{1}W(s)W^T(s)T_WXds\nonumber\\
	&&+W^T(t)T_WK^T_1T_W\int_{0}^{t}W(s)W^T(s)T_WXds\nonumber\\
	&&+W^T(t)T_WK^T_2T_W\int_{0}^{t}W(s)W^T(s)T_WX dB(s)\label{Laxman}	
\end{eqnarray}

Now
\begin{eqnarray*}
	&&	\int_{0}^{t}W(s)W^T(s)T_WXds\\
	&&=\int_{0}^{t}T_W\Phi(s)\Phi^T(s)T_WT_WXds\\
	&&	=mT_W\tilde{X}P\frac{1}{m}T_WW(t).
\end{eqnarray*}

Hence
\begin{equation}
	\int_{0}^{t}W(s)W^T(s)T_WXds=T_W\tilde{X}PT_WW(t)\label{int}
\end{equation}
Similarly,
\begin{equation}
	\int_{0}^{t}W(s)W^T(s)T_WXdB(s)=mT_W\tilde{X}P_S\frac{1}{m}T_WW(t)=T_W\tilde{X}P_ST_WW(t)\label{intS}
\end{equation}
Substituting \eqref{int} and \eqref{intS} in \eqref{Laxman} and using the condition of orthonormality, we get
\begin{eqnarray*}
	X^TT_WW(t)&=&F^TT_WW(t)+mW^T(t)T_WK^TX\\
	&&+mW^T(t)T_WK^T_1\tilde{X}PT_WW(t)\\
	&&+mW^T(t)T_WK^T_2\tilde{X}P_ST_WW(t)\\
	&=&F^TT_WW(t)+mW^T(t)T_WK^TX\\
	&&+W^T(t)T_WH_1T_WW(t)\\
	&&+W^T(t)T_WH_2T_WW(t)\\
	&=&F^TT_WW(t)+mW^T(t)T_WK^TX\\
	&&+m\hat{H_1}^TT_WW(t)+m\hat{H_2}^TT_WW(t)\\
\end{eqnarray*}	
which implies that, 
\begin{equation}
	\Big((I-mK)X^T-F^T-m\hat{H_1}^T-m\hat{H_2}^T\Big)T_WW(t)=0 \label{Eq:Final1}
\end{equation}
where $H_1=mK^T_1\tilde{X}P$, $H_2=mK^T_2\tilde{X}P_S$ and $\hat{H_i}$ is the $m$- vector with elements equal to the diagonal elements of $H_i$.\\
Hence 
\begin{equation}
	\Big((I-mK^T)X-F-m\hat{H_1}-m\hat{H_2}\Big)=[0]_{m\times 1}\label{Krishna}
\end{equation}
can be solved to obtain a non trivial solution of the given Stochastic Volterra- Fredholm integral equation \eqref{Eq:Ram}.

\section{Error Analysis} 
In this section, we analyse the error between the approximate solution and the exact solution of the stochastic Volterra- Fredholm integral equation. Before we start the analysis let us define, $\|X\|_2=E(|X|^2)^\frac{1}{2}$.\\
\begin{theorem}\label{fin}\cite{Paikaray}
	If $f\in L^2[0,1)$ satisfies the Lipschitz condition with Lipschitz constant $C$, then $\|e_m(t)\|_2=O(h)$, where $e_m(t)=|f(t)-\sum_{i=0}^{m-1}c_iw_i(t)|$ and  $c_i=\int_{0}^{1}f(s)w_i(s)ds$.  
\end{theorem}

\begin{theorem}\label{Thk}\cite{Paikaray}
	Suppose $k\in L^2\big([0,1)\times [0,1)\big)$ satisfies the Lipschitz condition with Lipschitz constant $L$. If  $k_m(x,y)=\sum_{i=0}^{m-1}\sum_{j=0}^{m-1}c_{ij}w_i(x)w_j(y)$, $c_{ij}=\int_{0}^{1}\int_{0}^{1}k(s,t)w_i(s)w_j(t)dtds$, then $\|e_m(x,y)\|_2=O(h)$, where $|e_m(x,y)|=|k(x,y)-k_m(x,y)|$.
\end{theorem}
\begin{theorem}
	Suppose $x_m(t)$ be the approximate solution of the linear SFVIE \eqref{Eq:Ram}. If
	\begin{enumerate}
		\item $f\in L^2[0,1)$, $k(s,t),k_1(s,t)\quad \text{and} \quad k_2(s,t)\in L^2\big( [0,1)\times[0,1)\big)$  satisfies the Lipschitz condition with Lipschitz constants  $C$, $L$, $L_1$ and $L_2$ respectively,
		\item  $|x(t)|\le \sigma$, $|k(s,t)|\le \rho$, $|k_1(s,t)|\le \rho _1$ and $|k_2(s,t)|\le \rho_2$
	\end{enumerate}
	then$$ \|x(t)-x_m(t)\|_2^2=O(h^2)$$

\end{theorem}
\begin{proof}
	Let \eqref{Eq:Ram} be the given SVFIE and $x_m(t)$ be the approximation to the solution using the Walsh function. 
	
	Then
	\begin{eqnarray*}
		x(t)-x_m(t)&=&f(t)-f_m(t)\\
		&+&\int_{\alpha}^{\beta}\big(k(s,t)x(s)-k_{m}(s,t)x_m(s)\big)ds\\
		&+&\int_{0}^{t}\big(k_1(s,t)x(s)-k_{1m}(s,t)x_m(s)\big)ds\\
		&+&\int_{0}^{t}\big(k_2(s,t)x(s)-k_{2m}(s,t)x_m(s)\big)dB(s)
	\end{eqnarray*}
	that implies,
	\begin{eqnarray*}
		|x(t)-x_m(t)|&\le&| f(t)-f_m(t)|\\\nonumber
		&+&\biggl|\int_{\alpha}^{\beta}\big(k(s,t)x(s)-k_{m}(s,t)x_m(s)\big)ds\biggr|\\
		&+&\biggl|\int_{0}^{t}\big(k_1(s,t)x(s)-k_{1m}(s,t)x_m(s)\big)ds\biggr|\\\nonumber
		&+&\biggl|\int_{0}^{t}\big(k_2(s,t)x(s)-k_{2m}(s,t)x_m(s)\big)dB(s)\biggr|.\nonumber
	\end{eqnarray*}
	We know that, $(a+b+c+d)^2\le 7a^2+7b^2+7c^2+7d^2$. Hence, 
	\begin{eqnarray}\label{eq:In}
		E\big(|x(t)-x_m(t)|^2\big)&\le&7E\biggl(| f(t)-f_m(t)|^2\biggr)\\\nonumber
		&+&7E\biggl(\biggl|\int_{\alpha}^{\beta}\big(k(s,t)x(s)-k_{m}(s,t)x_m(s)\big)ds\biggr|^2\biggr)\\\nonumber
		&+&7E\biggl(\biggl|\int_{0}^{t}\big(k_1(s,t)x(s)-k_{1m}(s,t)x_m(s)\big)ds\biggr|^2\biggr)\\\nonumber
		&+&7E\biggl(\biggl|\int_{0}^{t}\big(k_2(s,t)x(s)-k_{2m}(s,t)x_m(s)\big)dB(s)\biggr|^2\biggr).\nonumber
	\end{eqnarray}
	Now for $i=1,2$, we have
	\begin{eqnarray*}
		|k_i(s,t)x(s)-k_{im}(s,t)x_m(s)|
		\le&&|k_i(s,t)||x(s)-x_m(s)|\\
		&+&|k_i(s,t)-k_{im}(s,t)||x(s)|\\
		&+&|k_i(s,t)-k_{im}(s,t)||x(s)-x_m(s)|\\ 	
	\end{eqnarray*}
	For $i=1,2$, let $|k_i(s,t)|\le\rho_i$, $|x(s)|\le\sigma$ and using Theorem \ref{Thk}, we get
	\begin{equation}\label{normkernel}
		|k_i(s,t)x(s)-k_{im}(s,t)x_m(s)| \le \sqrt{2}L_ih\sigma+ (\rho_i+\sqrt{2}L_ih)|x(t)-x_m(t)|
	\end{equation}
	which gives,
	\begin{eqnarray*}
		E\biggl(\biggl|\int_{0}^{t}\big(k_1(s,t)x(s)-k_{1m}(s,t)x_m(s)\big)ds\biggr|^2\biggr)
		\le E\biggl(\biggl(\int_{0}^{t}\biggl|k_1(s,t)x(s)-k_{1m}(s,t)x_m(s)\biggr|ds\biggr)^2\biggr)\\
		\le E\biggl(\biggl(\int_{0}^{t}\big( \sqrt{2}L_ih\sigma+ (\rho_i+\sqrt{2}L_ih)|x(t)-x_m(t)|\big)ds\biggr)^2\biggr)
	\end{eqnarray*}
	By Cauchy- Schwarz inequality, for $t>0$ and $f\in L^2[0,1)$
	$$\biggl|\int_{0}^{t}f(s)ds\biggr|^2\le t\int_{0}^{t}|f|^2ds$$
	this implies,
		$$
		E\biggl(\biggl|\int_{0}^{t}\big(k_1(s,t)x(s)-k_{1m}(s,t)x_m(s)\big)ds\biggr|^2\biggr)
		\le 2E\biggl(\int_{0}^{t}\bigl((\sqrt{2}L_1h\sigma)^2+(\rho_1+\sqrt{2}L_1h)^2|x(t)-x_m(t)|^2
		\bigr)ds\biggr)
		$$
	
	Therefore,
	\begin{eqnarray}\label{eq:ink}
		E\biggl(\biggl|\int_{0}^{t}\big(k_1(s,t)x(s)-k_{1m}(s,t)x_m(s)\big)ds\biggr|^2\biggr)
		&\le&2(\sqrt{2}L_1h\sigma)^2\\\nonumber
		&+&2(\rho_1+\sqrt{2}L_1h)^2E\biggl(\int_{0}^{t}|x(t)-x_m(t)|^2ds\biggr)
	\end{eqnarray}
	Similarly, for $|k(s,t)|\le \rho$ and using Theorem \ref{Thk}, we get
	\begin{eqnarray}\label{eq:inkf}
		E\biggl(\biggl|\int_{\alpha}^{\beta}\big(k(s,t)x(s)-k_{m}(s,t)x_m(s)\big)ds\biggr|^2\biggr)
		&\le& 2(\beta -\alpha)(\sqrt{2}Lh\sigma)^2\\\nonumber
		&+&2(\rho+\sqrt{2}Lh)^2E\biggl(\int_{\alpha}^{\beta}|x(t)-x_m(t)|^2ds\biggr)
	\end{eqnarray}
	Now,
	\begin{eqnarray*}
		E\biggl(\biggl|\int_{0}^{t}\big(k_2(s,t)x(s)-k_{2m}(s,t)x_m(s)\big)dB(s)\biggr|^2\biggr)
		\le	E\biggl(\int_{0}^{t}\biggl|k_2(s,t)x(s)-k_{2m}(s,t)x_m(s)\biggr|^2ds\biggr)\\
		\le 2E\biggl(\int_{0}^{t}\big((\sqrt{2}L_2h\sigma)^2+(\rho_2+\sqrt{2}L_2h)^2|x(t)-x_m(t)|^2\big)ds\biggr)
	\end{eqnarray*}
	Hence,
	\begin{eqnarray}\label{eq:inkb}
		E\biggl(\biggl|\int_{0}^{t}\big(k_2(s,t)x(s)-k_{2m}(s,t)x_m(s)\big)dB(s)\biggr|^2\biggr)
		&\le&2(\sqrt{2}L_2h\sigma)^2\\\nonumber
		&+&2(\rho_2+\sqrt{2}L_2h)^2E\biggl(\int_{0}^{t}|x(t)-x_m(t)|^2ds\biggr)	
	\end{eqnarray}
	
	Using Theorem \ref{fin}, equation \eqref{eq:ink}, \eqref{eq:inkf} and \eqref{eq:inkb} in \eqref{eq:In}, we get
	\begin{eqnarray*}
		E\big(|x(t)-x_m(t)|^2\big)&\le&7C^2h^2\\
		&+&7\biggl( 2(\beta -\alpha)(\sqrt{2}Lh\sigma)^2+2(\rho+\sqrt{2}Lh)^2E\biggl(\int_{\alpha}^{\beta}|x(t)-x_m(t)|^2ds\biggr)\biggr)\\
		&+&7\biggl(2(\sqrt{2}L_1h\sigma)^2+2(\rho_1+\sqrt{2}L_1h)^2E\biggl(\int_{0}^{t}|x(t)-x_m(t)|^2ds\biggr)\biggr)\\
		&+&7\biggl(2(\sqrt{2}L_2h\sigma)^2+2(\rho_2+\sqrt{2}L_2h)^2E\biggl(\int_{0}^{t}|x(t)-x_m(t)|^2ds\biggr)\biggr)
	\end{eqnarray*}

	\begin{eqnarray}
		E\big(|x(t)-x_m(t)|^2\big)&\le&R_1
		+R_2\int_{0}^{t}E\bigl(|x(s)-x_m(s)|^2\bigr)ds
	\end{eqnarray}
	where,
	$$R_1=7\biggl(C^2h^2+2(\beta-\alpha)(\sqrt{2}Lh\sigma)^2+2(\sqrt{2}L_1h\sigma)^2+2(\sqrt{2}L_2h\sigma)^2\biggr)$$ and
	$$R_2=7\biggl(2(\rho+\sqrt{2}Lh)^2+2(\rho_1+\sqrt{2}L_1h)^2+2(\rho_2+\sqrt{2}L_2h)^2\biggr)$$
	By using Gronwall's inequality, we have
	\begin{eqnarray}
		E\big(|x(t)-x_m(t)|^2\big)&\le&R_1\exp\biggl(\int_{0}^{t}R_2ds\biggr).
	\end{eqnarray}
	which implies that,
	\begin{equation}
		\|x(t)-x_m(t)\|_2^2	=E\big(|x(t)-x_m(t)|^2\big)\le R_1e^{R_2}=O(h^2)
	\end{equation}

	
\end{proof}
	\section{Numerical Examples}

To illustrate the method given in the above section,we consider following examples and compute the approximate solution.The computations are done using Matlab 2013a.

\begin{example}\cite{Khodabin}\label{1}
	Consider the following linear SVFIE,
	$$x(t)=f(t)+\int_{0}^{1}cos(s+t)x(s)ds+\int_{0}^{t}(s+t)x(s)ds+\int_{0}^{t}e^{-3(s+t)}x(s)dB(s)$$ where $s,t\in [0,1)$ in which $f(t)=t^2+sin(s+t)-2cos(1+t)-2sin(t)-\frac{7t^4}{12}+\frac{1}{40}B(t)$, $B(t)$ is a Brownian motion, and $x(t)$ is an unknown stochastic process defined on the probability space $(\Omega, F, P)$.
\end{example}

\begin{table}[]
	\centering
	\caption{Numerical result for m=32 and m=64 in Example \ref{1}}
	\begin{tabular}{c c c c| c c c}
		\hline
		\multicolumn{3}{c}{$m=2^5$}& & &{$m=2^6$} \\
		\hline
		$t$	&WFM &BPF\cite{Khodabin} &HWM\cite{Mohamaddi}  &WFM &BPF\cite{Khodabin} &HWM\cite{Mohamaddi}\\
		\hline
		0.1&0.0114759&0.0199110&0.0189403&0.0085404&0.0155137&0.0184610\\
		0.3&0.0839521&0:1174676&0.1026368&0.0998259&0.0583251&0.1033269\\
		0.5&0.3296197&0.2741207&0.2469981&0.3385104&0.2775350&0.2462734\\
		0.7&0.4891180&0.5144708&0.4624837&0.4933237&0.4886760&0.4644731\\
		0.9&0.7826759&0.7685722&0.7642845&0.8223408&0.8222331&0.7640509\\
		
	\end{tabular}
	
\end{table}
\begin{figure}
	\centering
	\caption{Example \ref{1}'s approximate solution for m=32 and m=64 }
	
	\includegraphics[width=1.\textwidth]{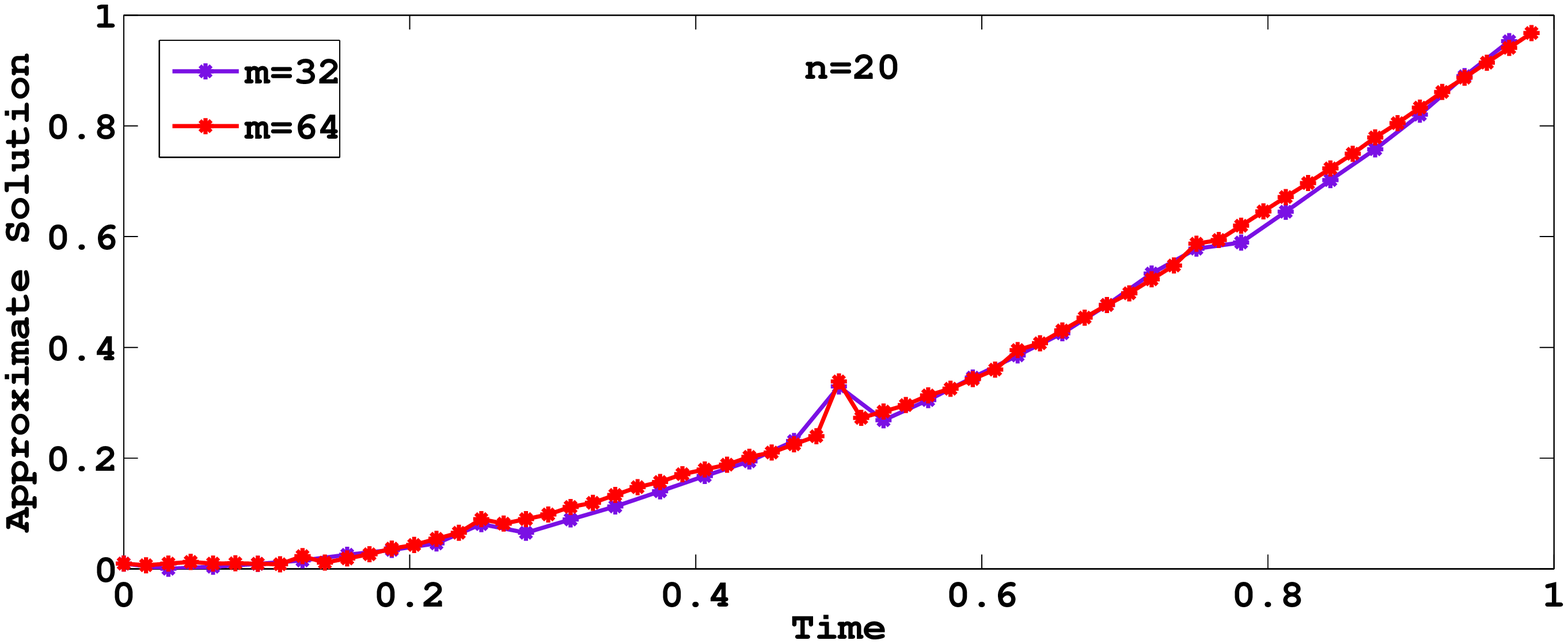}
	
	\label{fig:k=4}
\end{figure}
\begin{figure}
	\caption{Example \ref{2}'s approximate solution for m=32 and m=64 }
	\centering
	\includegraphics[width=1\textwidth]{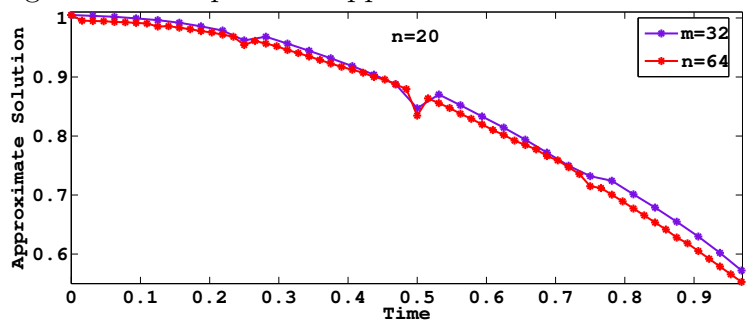}

\end{figure}
\begin{example}\cite{Mohamaddi}\label{2}
	Consider the following linear SVFIE,
	$$x(t)=f(t)+\int_{0}^{1}(s+t)x(s)ds+\int_{0}^{t}(s-t)x(s)ds+\frac{1}{125}\int_{0}^{t}sin(s+t)x(s)dB(s)$$ where $s,t\in [0,1)$ in which $f(t)=2-cos(1)-(1+t)sin(1)+\frac{1}{250}sin(B(t))$, $B(t)$ is a Brownian motion, and $x(t)$ is an unknown stochastic process defined on the probability space $(\Omega, F, P)$.
\end{example}

\begin{table}[ht]
	\centering
	\caption{Numerical result for m=32 and m=64 in Example \ref{2}}
	\begin{tabular}{c c c c| c c c}
		\hline
		\multicolumn{3}{c}{$m=2^5$}& & &{$m=2^6$} \\
		\hline
		$t$	&WFM &BPF\cite{Khodabin} &HWM\cite{Mohamaddi}  &WFM &BPF\cite{Khodabin} &HWM\cite{Mohamaddi}\\
		\hline
		0.1&0.9976241&0.9983232&0.9526175&0.9912432&0.9958677&0.9535115\\
		0.3&0.9592595&0.9427155&0.9044299&0.9510972&0.9618340&0.9058330\\
		0.5&0.8470106&0.8930925&0.8149461&0.8345253&0.8503839&0.8160360\\
		0.7&0.7669107&0.7695923&0.6922649&0.7610515&0.7566968&0.6943825\\
		0.9&0.6438552&0.6924411&0.5480265&0.6105657&0.6120356&0.5496713\\
	\end{tabular}
	
\end{table}

\section{Conclusion}
Since it is challenging to find the exact solution for a majority of the SVFIEs, the numerical technique is crucial in solving these issues. Several numerical solutions have also been developed earlier to determine the approximate solution of SVFIEs. This article also proposes a numerical method to find an approximate solution to SVFIE. It also includes numerical estimates for some SVFIEs. The important part is that error analysis of the approach has been undergone by considering the functions satisfing Lipschitz condition to confirm the validity of the methodology which gives an upper hand to consider more general SVFIEs than the previous method . This method can further be developed to address the non linear stochastic integral equations.  
%
%

\end{document}